\documentclass[reqno,12pt]{amsart}
\usepackage{amsmath, amsthm, amscd, amsfonts, amssymb, graphicx, color}
\usepackage[bookmarksnumbered, colorlinks, plainpages]{hyperref}

\theoremstyle{definition}

\theoremstyle{definition}
\theoremstyle{remark}

\newtheorem{thm}{\textbf{Theorem}}[section]

\newtheorem{de}[thm]{\textbf{Definition}}
\newtheorem{ex}[thm]{\textbf{Example}}
\newtheorem{rem}[thm]{\textbf{Remark}}
\newtheorem{co}[thm]{\textbf{Corollary}}

\numberwithin{equation}{section}

\title[Approximation of solution set of a variational inequality ]{ Approximation of solution set of a variational inequality   for $(u, v)$-cocoercive   mappings in Banach spaces}
\author{ Ebrahim  Soori }
 \thanks{ \!\!\!\!\!\!\!\! \!\!2010 Mathematics Subject Classification: 47H09; 47H10. \\ E-mail address: sori.e@lu.ac.ir; sori.ebrahim@yahoo.com.}

\begin{document}
\begin{large}


\maketitle


 \begin{center}
 \begin{normalsize}
    Department  of Mathematics, Lorestan University, Lorestan, Khoramabad, Iran.
 \end{normalsize}
 \end{center}
\begin{abstract}

\begin{normalsize}
Let $C$ be a nonempty closed   convex subset of a  real normed linear  space $E$ and $u, v$ are positive numbers. In this paper we introduce  some new definitions that   generalize the analogue   definitions from real Hilbert spaces to  real normed linear spaces.  Indeed, we generalize $(u, v)$-cocoercive mappings and $v$-strongly monotone mappings  and  $V I (C, B)$ for a   mapping $B$, from real  Hilbert spaces to  real normed linear spaces. Then we      prove that  the generalized variational inequality    $V I (C, B)$ is singleton for   $(u, v)$-cocoercive mappings     under appropriate assumptions on Banach spaces that extends and improves Propositions 2, 3 in [S. Saeidi, Comments on relaxed $(u, v)$-cocoercive mappings. Int. J. Nonlinear Anal. Appl. 1 (2010) No. 1, 54-57].
\end{normalsize}
\end{abstract}
\begin{normalsize}
   \textbf{keywords}:   Fixed point; Nonexpansive mapping; $(u, v)$-cocoercive mapping;  Duality mapping;   metric projection.
   \end{normalsize}

\section{ Introduction}
Let $C$ be a nonempty closed   convex subset of a real normed linear  space $E$ and
$E^{*}$ be the dual space of $E$.  Suppose that  $\langle.,.\rangle$   denote the pairing between $E$ and $E^{*}$. The
normalized duality mapping $J: E \rightarrow E^{*}$
is defined by
\begin{align*}
    J(x)=\{f \in E^{*}: \langle x, f \rangle= \|x\|^{2}=\|f\|^{2} \}
\end{align*}
for each $x \in E$.  Suppose that   $U = \{x \in E : \|x\| = 1\}$.   A Banach space $E$ is called  smooth if for all $x \in U$,
there exists a unique functional $j_{x} \in E^{*}$ such that $\langle x, j_{x}\rangle = \|x\|$ and $\|j_{x}\| = 1$ ( see \cite{Ag}).

Let $C$ be a nonempty closed   convex subset of a real normed linear  space $E$.  A mapping $ T$ of $ C $ into itself is said to be  nonexpansive if $\|Tx - Ty\| \leq \|x - y\|,$ for all $x, y \in C$ and a mapping $f$ is an $\alpha$-contraction on $E $ if  $ \|f (x) -f (y)\| \leq \alpha \|x - y\|, \;x, y \in E$   and $0 \leq\alpha < 1$.

 Let $C$ be a nonempty closed convex subset of a real Hilbert space $ H $.   Suppose that  $B : C \rightarrow H $ be a nonlinear map and  $P_{C} $ be the projection of $H$ onto   $C$. The classical variational inequality problem $V I (C, B)$ is to find $u \in C $ such that
 \begin{equation}\label{23}
    \langle Bu, v - u\rangle\geq 0,
 \end{equation}
 for all $v \in C$ (see \cite{sa}). For a given $z \in H$, $u \in C$ satisfies the inequality
 \begin{equation}\label{w}
    \langle u - z, v - u\rangle\geq 0,\quad  ( v \in C),
 \end{equation}
 if and only if $u = P_{C} z$. Therefore
\begin{align*}
u \in V I(C,B) \Longleftrightarrow u = P_{C}(u -\lambda Bu),
\end{align*}
where  $\lambda > 0 $ is a constant(see \cite{sa}).   It is known that the projection operator  $P_{C}$ is nonexpansive. It is also known that  $P_{C}$  satisfies
 \begin{equation}\label{24}
    \langle x - y, P_{C} x - P_{C} y\rangle \geq \|P_{C} x - P_{C} y\|^{2},
 \end{equation}
 for $x,y \in H $.

 Let $C$ be a nonempty closed convex subset of a real Hilbert space $ H $, recall   the following definitions (see \cite{sa}):\\
 \begin{enumerate}
   \item [(i)] $B$ is called $v$-strongly monotone if
 $$\langle Bx - By\;,\; x - y\rangle\geq v\|x - y\|^{2}\qquad \text{for\: all}\quad x, y \in C,\qquad\qquad\qquad\qquad$$
 for a constant $v > 0$.
   \item [(ii)] $B$ is said to be relaxed $(u, v)$-cocoercive, if there exist two constants $u, v > 0$ such that
 \begin{equation*}
    \langle Bx - By, x - y\rangle \geq (-u)\|Bx - By\|^{2}+v\|x - y\|^{2},
 \end{equation*}
   for  all  $x, y \in C$. For $u = 0$, $B$ is $v$-strongly monotone.   Clearly, every  $v$-strongly monotone map is a relaxed $(u, v)$-cocoercive map.
 \end{enumerate}

Let $C$ be a nonempty closed convex subset of a real normed linear  space $E$. In this paper we introduce    some new definitions that   generalize the analogue   definitions from Hilbert spaces to  real normed linear  spaces. Then we      prove  that $V I (C, B)$ is singleton for   $(u, v)$-cocoercive mappings  under appropriate assumptions on Banach spaces.

\section{preliminaries}
Suppose that  $C$ be a nonempty subset of a normed space $E$ and let $x \in E$. An element
$y_{0} \in C$ is said to be a best approximation to $x$ if
$\|x - y_{0}\| = d(x,C)$,
where
\begin{equation}\label{pcd}
    d(x,C) = \inf _{y\in C} \|x - y\|.
\end{equation}
  The number $d(x,C)$ is called the distance from $x$
to $C$.

The   set of all best approximations from $x$ to $C$ is denoted
by
$P_{C}(x) = \{y \in C : \|x - y\| = d(x,C)\}$.
This defines a mapping $P_{C}$ from $E$ into $2^{C}$ and is called the metric projection
onto $C$.   $C$ is Chebyshev if $P_{C}(x)$ is singleton for each $x \in E$ and $C$ is proximinal if $P_{C}(x)\neq\emptyset$, for all $x \in E$. Every closed convex subset $C$ of a reflexive Banach space is proximinal and
every closed convex subset $C$ of a reflexive strictly convex Banach space is a Chebyshev
set. Let $C$  be a proximinal subset of a Banach space $E$, by proposition 2.10.1  in \cite{Ag}
$C$  is closed, hence Chebyshev  subsets of a Banach space $E$ are closed, too(for more detail see page 115 in \cite{Ag}).

A continuous strictly increasing function \\$\mu : \mathbb{R}^{+} \rightarrow \mathbb{R}^{+}$ is
said to be gauge function if $\mu(0) = 0$ and $\lim _{t\rightarrow\infty}\mu(t) = \infty$.\\
Let $E$ be a normed space and $E^{*}$ be it's dual space.    Let  $\mu$  be a gauge function. Then the
mapping   $J_{\mu}: E \rightarrow E^{*}$
is defined by
\begin{align*}
    J_{\mu}(x)=\{j \in X^{*}: \langle x, j \rangle= \|x\|\|j\|_{*}\, , \, \|j\|_{*}= \mu(\|x\|) \},
\end{align*}
for all $x \in E$. $J_{\mu}$  is called the duality mapping with gauge function $\mu$.
In the particular case $\mu(t) = t$, the duality mapping $J_{\mu} = J$ is   the
normalized duality mapping \cite{Ag}.
\begin{thm}\cite{Ag}\label{ju}
Let $C$ be a nonempty convex subset of a smooth Banach
space $E  $ and let $x \in E$ and $y \in C$. Then the following are equivalent:\\
(a) $y$ is a best approximation to $x$: $\|x - y\| = d(x,C)$.\\
(b) $y$ is a solution of the variational inequality:\\
$\langle y - z, J_{\mu} (x - y)\rangle \geq 0$, for all $z \in C$,
where $J_{\mu}$ is a duality mapping with gauge function $\mu$.
\end{thm}

\section{New definitions and examples}

Let $ E $ be a real  normed linear space.   First, we  introduce the following   new definition:
\begin{de}
Let   $C$ be a nonempty  closed convex subset of a real normed linear space  $E$ and\\ $B : C \rightarrow E $ be a nonlinear map. $B$ is said to be relaxed $(u, v)$-cocoercive, if there exist two constants $u, v > 0$ such that
 \begin{equation*}
    \langle Bx - By, j(x - y)\rangle \geq (-u)\|Bx - By\|^{2}+v\|x - y\|^{2},
 \end{equation*}
   for  all $x,y\in C$ and $j(x - y)\in J(x-y)$.
\end{de}
\begin{ex}\label{ex2}
  Let $C$ be a nonempty closed convex subset of a real Hilbert space $ H $, it is well-known  that
 $B : C \rightarrow H $ is said to be relaxed $(u, v)$-cocoercive, if there exist two constants $u, v > 0$ such that
 \begin{equation*}
    \langle Bx - By, x - y\rangle \geq (-u)\|Bx - By\|^{2}+v\|x - y\|^{2},
 \end{equation*}
 for  all  $x, y \in C$.  By Example 2.4.2 in \cite{Ag}, in a Hilbert space $ H $, the normalized duality mapping is the
identity.     Then $J(x-y)=\lbrace x-y\rbrace  $. Therefore, the above definition  extends   the definition of  relaxed $(u, v)$-cocoercive mappings, from real Hilbert spaces   to real  normed linear  spaces.
  \end{ex}

 Let us to define $v$-strongly monotone mappings on real normed linear spaces, too.
 \begin{de}
Let   $C$ be a nonempty  closed convex subset of a real normed linear space  $E$ and\\ $B : C \rightarrow E $ be a nonlinear map. $B$  is called $v$-strongly monotone   if there exists     a constant $v > 0$ such that
 $$\langle Bx - By\;,\; j(x - y)\rangle\geq v\|x - y\|^{2},$$
   for  all $x,y\in C$ and $j(x - y)\in J(x-y)$.
\end{de}
\begin{ex}\label{ex3}
  Let $C$ be a nonempty closed convex subset of a real Hilbert space $ H $, it is well-known, too,  that
 $B : C \rightarrow H $ is said to be $v$-strongly monotone, if there exists  a constant $v > 0$ such that
 \begin{equation*}
    \langle Bx - By, x - y\rangle\geq v\|x - y\|^{2},
 \end{equation*}
 for  all  $x, y \in C$.
 Since $ H $ is  a Hilbert space,\\ $J(x-y)=\lbrace x-y\rbrace$. Therefore, the above definition  extends   the definition of  $v$-strongly monotone mappings, from real  Hilbert spaces to  real  normed linear spaces.
 \end{ex}

\begin{ex}\label{zendegi}
 Let $C$ be a nonempty closed convex subset of a real  Banach space $E$. Let $T$ be an $\alpha$-contraction of  $C$ into itself. Putting  $B=I-T$, we have
\begin{align*}
    \langle Bx - By,& j(x - y)\rangle\\ =& \langle (I-T)x - (I-T)y, j(x - y)\rangle\\=&
 \langle (x-y) - (Tx-Ty), j(x - y)\rangle\\=&   \langle x-y , j(x - y)\rangle- \langle Tx-Ty , j(x - y)\rangle\\\geq &
 \langle x-y , j(x - y)\rangle - \| Tx-Ty \| \|j(x - y)\|\\\geq &
 \| x-y \|^{2} - \| Tx-Ty \| \|x - y\|\\\geq &  \| x-y \|^{2} -\alpha \| x-y \|^{2}=(1-\alpha)\| x-y \|^{2},
\end{align*}
 for  all $x,y\in C$ and $j(x - y)\in J(x-y)$. Hence \\ $B : C \rightarrow E $ is a $(1-\alpha)$-strongly monotone mapping,  therefore $B$ is a relaxed $(u, (1-\alpha))$-cocoercive mapping on $E$ for each  $u >0 $.
\end{ex}
 Now, we can  concider  the following   definition that generalizes  the classical variational inequality problem  \ref{23}.
\begin{de}
 Let $ E $ be a real normed linear space and $C$ be a nonempty  closed convex subset of   $E$. Let \\$B : C \rightarrow E $ be a nonlinear map.    The      classical variational inequality problem $V I (C, B)$  is to find $u \in C $ such that
 \begin{equation}\label{vi}
    \langle j(Bu)\;,\;  v - u \rangle\geq 0,
 \end{equation}
 for all $v \in C$ and $j(Bu) \in J(Bu)$.
\end{de}
\begin{ex}\label{ex4}
 Let $C$ be a nonempty closed convex subset of a real Hilbert space $ H $     and
 $B : C \rightarrow H $   be a   mapping. Since $ H $ is a   Hilbert space,   $J(Bu)=\lbrace Bu\rbrace  $. Therefore,   \ref{vi}  generalizes    \ref{23} from   real  Hilbert spaces to real normed linear spaces.
 \end{ex}
 \begin{rem}\label{cheb}
Let $C$ be a nonempty   convex Chebyshev subset of a  real smooth Banach space $E$. Putting  $\mu(t)=t$, from   Theorem  \ref{ju}, we have
\begin{align}\label{aga}
u \in V I(C,B) \Longleftrightarrow u = P_{C}(u -\lambda Bu).
\end{align}
\end{rem}

Now, we introduce  the following new definition,
\begin{de}
   Let $C$ be a nonempty  Chebyshev    subset of  a   normed linear space  $E$ such that  $P_{C}$ be  a  metric  projection  from $E$ into $C$. Let
   $B$  be a   mapping from $C$ into $E$.  $B$ is said to be $P_{C}$-nonexpansive, if    $\|P_{C}Bx-P_{C}By\|\leq  \| Bx- By \|$.
\end{de}

\begin{ex}\label{ex5}
  Let   $C$  be a nonempty closed convex subset of    a Hilbert space $H$ and $P_{C}$, the metric projection from $H$ onto $C$ and  $B  $ a   mapping from $C$ into $H$, since by Proposition 2.10.15  in \cite{Ag},  $P_{C}$ is a nonexpansive projection, we have $\|P_{C}Bx-P_{C}By\|\leq  \| Bx- By \|$, therefore $B $ is $P_{C}$-nonexpansive.
\end{ex}

\begin{rem}
 In a Banach space, a metric projection mapping is  not nonexpansive, in general. However, the existence of nonexpansive projections from a Banach space even into a nonconvex subset $\Omega$,  is discussed  in  \cite{br}. Let $C$ is a Chebyshev  subset of a Banach space $E$ and  $B  $ a   mapping from $C$ into $E$. Therefore if a metric projection $P_{C}$ from a Banach space   into       $C$ is nonexpansive, then we have $\|P_{C}Bx-P_{C}By\|\leq  \| Bx- By \|$, therefore $B $ is $P_{C}$-nonexpansive.
\end{rem}
\begin{ex}
 Let
 $C$ be a nonempty closed convex subset of a strictly convex  and
reflexive Banach space $E$.    By Corollary 2.10.11 in \cite{Ag}, there exists a metric projection mapping $P_{C}: X \rightarrow C$ such that $P_{C}(x) = x$ for all $x \in C$.  Let  $B  $ is  a   mapping from $C$ into $C$,  therefore we have   $\|P_{C}Bx-P_{C}By\|=   \| Bx- By \|$, hence $B $ is $P_{C}$-nonexpansive.
\end{ex}
\section{Main results}
In this section, we deal with some results to prove that $V I(C,B)$ is singleton when $B : C \rightarrow E$
is a relaxed $(u, v)$-cocoercive and  $ 0<\mu$-Lipschitzian mapping and  $C$ is a nonempty convex Chebyshev   subset of a real smooth Banach space $E$.

\begin{thm}\label{jmen} Let $E$ be a Banach space, for all $x,y \in E$, we have
   $$ \langle x-y\;,\; j(x-y)\rangle\leq \langle x-y\;,\;  x^{*}-y ^{*}\rangle+4\|x\|\|y\|,$$
   for all $x^{*}\in J(x),  y ^{*} \in J(y), j(x-y) \in J(x-y)$.
\end{thm}
\begin{proof}
   Let $x= y$, obviously the inequality holds. Let  $x^{*}\in J(x),  y ^{*} \in J(y) $ and $x\neq y$. As in the proof of Theorem 4.2.4 in \cite{tn}, we have
    \begin{align*}
        \langle x-y\;,\;  x^{*}-&y ^{*}\rangle \\ \geq & (\|x\|-\|y\|)^{2}\\&+ (\|x\|+\|y\|)(\|x\|+\|y\|-\|x+y\|).
    \end{align*}
    Hence, we have
    \begin{align*}
        \langle x-y\;,\;  x^{*}-&y ^{*}\rangle \\ \geq & (\|x\|-\|y\|)^{2}\\&+ (\|x\|+\|y\|)(\|x\|+\|y\|-\|x+y\|)\\=& (\|x\|-\|y\|)^{2}\\&+ (\|x\|+\|y\|)^{2}-\|x+y\| (\|x\|+\|y\|)\\ \geq &
        (\|x\|-\|y\|)^{2}\\&+ \|x-y\|^{2}- (\|x\|+\|y\|)^{2}
         \\ = & \|x-y\|^{2}-4\|x\|\|y\|\\=&
        \langle x-y\;,\; j(x-y)\rangle -4\|x\|\|y\|,
        \end{align*}
        therefore,
           $$ \langle x-y\;,\; j(x-y)\rangle\leq \langle x-y\;,\;  x^{*}-y ^{*}\rangle+4\|x\|\|y\|.$$
\end{proof}

\begin{thm}\label{thm1}
  Let $C$ be a nonempty   convex Chebyshev   subset of a real smooth Banach space $E$.  Suppose that $\mu, v, u$ be real numbers such that $\mu>0 $,  and $v > u\mu^{2}+5\mu$. Let  $B: C \rightarrow E$
be a relaxed $(u, v)$-cocoercive and  $ \mu$-Lipschitzian mapping. Let $P_{C}$ be a metric projection mapping from $E$ into $C$  such that $I-\lambda B $ be a $P_{C}$-nonexpansive mapping,  for each $\lambda>0$.
    Then
$V I(C,B)$ is singleton.
\end{thm}
\begin{proof}
Let $\lambda$ be a real number such that
\begin{equation}
    0<\lambda<\frac{v-u\mu^{2}-5\mu}{\mu^{2}}, \quad \lambda   \mu^{2}[\frac{v-u\mu^{2}-5\mu}{\mu^{2}}- \lambda ]<1.
\end{equation}
    Then, by Theorem \ref{jmen},  for every $x, y \in C$, we have
    \begin{align*}
        \|P_{C}&(I -\lambda B)x- P_{C}(I -\lambda B)y\|^{2}\\ \leq &
        \|(I -\lambda B)x- (I -\lambda B)y\|^{2}\\ = &
        \|(x-y)-  \lambda (Bx- By)\|^{2}\\ = &
        \|j[(x-y)-  \lambda (Bx- By)]\|^{2}\\ = &
        \langle (x-y)-  \lambda (Bx- By),j[(x-y)\\&-  \lambda (Bx- By)]\rangle \\ \leq &
        \langle x-y -\lambda (Bx- By)  \;,\; j(x-y)\\&-\lambda j(Bx- By)\rangle\\&+4\lambda \|x-y \| \|Bx- By \|\\=&
        \langle x-y   \;,\; j(x-y)\rangle\\&-\lambda \langle  Bx- By   \;,\; j(x-y)\rangle \\&+\lambda \langle y-x  \;,\;   j(Bx- By)\rangle \\&+ \lambda ^{2} \langle  (Bx- By)  \;,\;  j(Bx- By)\rangle \\&+4\lambda \|x-y \| \|Bx- By \|\\\leq &
        \|x-y\|^{2}+\lambda u  \|Bx- By\|^{2}- \lambda v \|x-y\|^{2}\\&+ \lambda^{2} \|Bx- By\|^{2}+5\lambda \|x-y \| \|Bx- By \|\\ \leq &
        \|x-y\|^{2}+\lambda u  \mu^{2}\|x- y\|^{2}- \lambda v \|x-y\|^{2}\\&+ \lambda^{2}\mu^{2} \|x-y\|^{2}+5\lambda \mu \|x-y \| ^{2}\\ \leq &
        \Big(1+\lambda u  \mu^{2}- \lambda v+ \lambda^{2}\mu^{2} +5\lambda \mu \Big)\|x-y\| ^{2}\\\leq &
        \Big(1-\lambda   \mu^{2}[\frac{v-u\mu^{2}-5\mu}{\mu^{2}}- \lambda ] \Big)\|x-y\| ^{2}
    \end{align*}
    Now, since $1-\lambda   \mu^{2}[\frac{v-u\mu^{2}-5\mu}{\mu^{2}}- \lambda ] < 1$, the mapping \\$P_{C}(I -\lambda B) : C \rightarrow C$ is a contraction
and Banach's Contraction Mapping Principle guarantees that it has a unique fixed
point  $  u$; i.e., $P_{C}(I -\lambda B)u = u$, which is the unique solution of $V I(C,B)$ by  \ref{aga}.
\end{proof}
  We can conclude  Proposition 2 in \cite{sa} for $v > u\mu^{2}+5\mu$.
\begin{co}([4, Proposition 2]) Let $C$ be a nonempty closed convex subset of a Hilbert space  $H$ and let \\$B : C \rightarrow H$
be a relaxed $(u, v)$-cocoercive and  $ 0<\mu$-Lipschitzian mapping such that $v > u\mu^{2}+5\mu$. Then
$V I(C,B)$ is singleton.

\end{co}
\begin{rem}
S.  Saeidi, in the proof of  Proposition 2  in  \cite{sa}          proves  that
\begin{align*}
 \|P_{C}&(I -  s A)x- P_{C}(I -s A)y\|^{2} \\& \leq  \Big(1-s   \mu^{2}[\frac{2(r-\gamma \mu^{2})}{\mu^{2}}- s ] \Big)\|x-y\| ^{2}
\end{align*}
when $ 0< s<\frac{2(r-\gamma \mu^{2})}{\mu^{2}}$ and $ r>\gamma \mu^{2}$. Putting \\$ r=\gamma=s=1 $ and $ \mu=\frac{1}{10} $ we have \\$\Big(1-s   \mu^{2}[\frac{2(r-\gamma \mu^{2})}{\mu^{2}}- s ] \Big)<0  $ that is a contradiction. We correct  this contradiction in the proof of theorem \ref{thm1}.

\end{rem}
Since $v$-strongly monotone mappings are relaxed $(u, v)$-cocoercive, we conclude  the following theorem:
\begin{thm}\label{thm2}
    Let $C$ be a nonempty   convex Chebyshev   subset of a real smooth  Banach space $E$.  Suppose that $\mu, v$ be real numbers such that $\mu>0 $  such that  $v > 5\mu$.
Let  $B: C \rightarrow E$
be a  $v$-strongly monotone,  $ \mu$-Lipschitzian mapping. Let $P_{C}$ be a metric projection mapping from $E$ into $C$  such that $I-\lambda B $ be a $P_{C}$-nonexpansive mapping,  for each $\lambda>0$.
    Then
$V I(C,B)$ is singleton.
\end{thm}
We can conclude  Proposition 3 in \cite{sa} for $v > 5\mu$:
\begin{co}([4, Proposition 3]) Let $C$ be a nonempty closed convex subset of a Hilbert space  $H$ and let \\$B : C \rightarrow H$
be a  $v$-strongly monotone and $ 0<\mu$-Lipschitzian mapping such that  $v > 5\mu$. Then
$V I(C,B)$ is singleton.

\end{co}
\end{large}


\bigskip
\bigskip


\end{document}